\documentclass{amsart}
\usepackage{amsmath,amsfonts,latexsym,amssymb}

\newtheorem{theorem}{Theorem}
\newtheorem{lemma}{Lemma}

\theoremstyle{remark}

\begin{document}

\markboth{Ritabrata Munshi}{The circle method and bounds for $L$-functions - I}
\title[The circle method and bounds for $L$-functions - I]{The circle method and bounds for $L$-functions - I}

\author{Ritabrata Munshi}   
\address{School of Mathematics, Tata Institute of Fundamental Research, 1 Homi Bhabha Road, Colaba, Mumbai 400005, India.}     
\email{rmunshi@math.tifr.res.in}

\begin{abstract}
Let $f$ be a Hecke-Maass or holomorphic primitive cusp form of arbitrary level and nebentypus, and let $\chi$ be a primitive character of conductor $M$. For the twisted $L$-function $L(s,f\otimes \chi)$ we establish the hybrid subconvex bound
$$
L\left(\frac{1}{2}+it,f\otimes\chi\right)\ll (M(3+|t|))^{\frac{1}{2}-\frac{1}{18}+\varepsilon},
$$ 
for $t\in \mathbb R$. The implied constant depends only on the form $f$ and $\varepsilon$.
\end{abstract}


\maketitle


\section{Introduction}
\label{intro}

\subsection{Statement of result}
In this paper we prove the following hybrid subconvex bound:
\begin{theorem}
Let $f$ be a Hecke-Maass or holomorphic primitive cusp form for $\Gamma_0(P)\subset SL(2,\mathbb Z)$ with nebentypus $\psi$. Let $\chi$ be a primitive Dirichlet character of modulus $M$. Then we have
$$
L\left(\tfrac{1}{2}+it,f\otimes\chi\right)\ll_{f,\varepsilon} (MT)^{\frac{1}{2}-\frac{1}{18}+\varepsilon},
$$
where $T=3+|t|$.
\end{theorem}

Our result beats the convexity bound $L\left(\tfrac{1}{2}+it,f\otimes\chi\right)\ll_{f,\varepsilon} (MT)^{\frac{1}{2}+\varepsilon}$, simultaneously in the $t$ and $M$ aspect. This is one of the most extensively studied subconvexity problems in the literature. In the $t$-aspect, subconvexity was accomplished by Good \cite{G} (in the holomorphic case of full level) and Meurman \cite{Me-2} (in the case of Maass forms of full level). For number fields, $t$-aspect subconvexity was achieved by Cogdell, Piatetski-Shapiro and Sarnak \cite{CPS}, Petridis and Sarnak \cite{PS} and Diaconu and Garrett \cite{DG}. The first subconvex bound in the $M$ aspect (for $t=0$) was obtained by Duke, Friedlander and Iwaniec \cite{DFI-1} (for holomorphic forms of full level), Bykovskii \cite{By} (for general holomorphic forms) and later by Harcos \cite{H}, Michel \cite{Mi} and Blomer, Harcos and Michel \cite{BHM} (for Maass forms). Over number fields this was established by Venkatesh \cite{V} and Blomer and Harcos \cite{BH-2}. \\

The first hybrid subconvex bound was given by Blomer and Harcos \cite{BH}. They obtained 
$$
L\left(\tfrac{1}{2}+it,f\otimes\chi\right)\ll (MT)^{\frac{1}{2}-\frac{1}{40}+\varepsilon},
$$
for $f$ Maass or holomorphic form of general level. They however only tackle the case of trivial nebentypus. In \cite{M} we established the weaker exponent $\frac{1}{2}-\frac{1}{118}$, for holomorphic $f$ of general level and general nebentypus. For number fields, Michel and Venkatesh \cite{MV} have established such a hybrid subconvex bound with an unspecified exponent.\\

At present we have the Weyl-type bound in the $t$-aspect, and the Burgess-type bound in the $M$-aspect \cite{BH-2} over general number fields, under the Ramanujan conjecture. (This is expected to be the natural boundary of the current methods.) To complete the story one seeks to prove a Burgess-type hybrid bound over general number fields (assuming Ramanujan conjecture). But this has not been achieved yet, not even over $\mathbb Q$. In the context of subconvexity, the only Burgess-type hybrid bound known is due to Heath-Brown \cite{HB} in the case of Dirichlet $L$-function $L(s,\chi)$.

\subsection{Sketch of proof}
We will briefly describe our method, which we believe to be new. The experts will have no problem in filling the gaps. We start with the approximate functional equation which reduces the problem to getting cancellation in the sum 
$$
\sum_{n\sim N}\lambda_f(n)\chi(n)n^{-it},
$$
for $N\ll (MT)^{1+\varepsilon}$, where $\lambda_f(n)$ are the normalized Fourier coefficients of $f$. Let us consider the worst case scenario namely $N=MT$. We write this sum as
$$
\mathop{\sum\sum}_{n,m\sim N}\lambda_f(n)\chi(m)m^{-it}\delta_{n,m}
$$
where $\delta_{n,m}$ is the Kronecker symbol. Then we apply additive harmonics - the circle method, to detect the equation $n=m$ (see Section \ref{setup}). Suppose we use the $\delta$-method of Duke, Friedlander and Iwaniec. Roughly speaking this will yield
$$
\frac{1}{Q^2}\sum_{q\leq Q}\;\sideset{}{^\star}\sum_{a\bmod{q}}\mathop{\sum\sum}_{n,m\sim N}\lambda_f(n)\chi(m)m^{-it}e_q(an-am)
$$
where $e_q(x)=e^{2\pi ix/q}$ and $Q=\sqrt{N}=\sqrt{MT}$.\\ 

Next we apply the Poisson summation on the sum over $m$ and the Voronoi summation on the sum over $n$ (see Section \ref{gl1_twists}). This will reduce the above sum roughly to (see Lemma \ref{sum-form})
$$
\sum_{q\leq \sqrt{MT}}\mathop{\sum\sum}_{\substack{|n|\ll 1\\|m|\ll \sqrt{MT}}}\lambda_f(n)\bar\chi(m)\chi(q)m^{it}q^{-it}e_q(-M\bar m n).
$$
(Here one needs to use stationary phase to deal with some exponential integrals.) Trivial estimation of this sum yields the convexity bound. Any further saving will give us subconvexity. However it is not useful to execute the sum over $q$, as the modulus of the sum (after reciprocity) is of size $(MT)^{3/2}$, whereas the length is just $\sqrt{MT}$. As in \cite{M} we can now use the bound for bilinear forms with Kloosterman fractions, and get a subconvexity result with exponent $\frac{1}{2}-\frac{1}{118}$. However we will avoid this result, and as an alternative use a very simple idea to get a better exponent.\\

Suppose the collection of moduli $q$ we are using in the circle method, has a multiplicative structure - namely each $q$ factorizes uniquely as $q=q_1q_2$, with $q_1\leq Q_1$ and $q_2\leq Q_2$ (with $Q_1Q_2=\sqrt{MT}$). Then applying Cauchy to the above sum one arrives at
$$
\sqrt{Q_2}(MT)^{\frac{1}{4}}\left[\sum_{q_2\leq Q_2}\mathop{\sum\sum}_{\substack{|n|\ll 1\\|m|\ll \sqrt{MT}}}\left|\sum_{q_1\leq Q_1}\chi(q_1)q_1^{-it}e_{q_1q_2}(-M\bar m n)\right|^2\right]^{\frac{1}{2}}.
$$ 
Again for subconvexity we just need some cancellation in the remaining sum. We see that we need $Q_1$ to have some size, because that is exactly the amount we save in the diagonal. For the off-diagonal we will again apply Poisson on the sum over $m$. We save in the off-diagonal as long as the modulus which is of the size $Q_1^2Q_2=(MT)/Q_2$ is smaller than the square of the length of the $m$-sum, which is $\sqrt{MT}$. So the off-diagonal will be satisfactory if $Q_2$ has some size. (We explain this in Section \ref{conclusion}.)\\

Of course to get an inbuilt bilinear structure in the circle method itself, we need to use a more flexible version of the circle method - the one investigated by Jutila. This version comes with an error term which is satisfactory, as we shall find out, as long as we allow the moduli to be slightly larger than $\sqrt{MT}$ (see Section~\ref{setup}). For another application of this idea see \cite{M2}. Further applications of this idea in the context of subconvexity will be given in a follow up paper.


\section{Preliminaries}
\label{prelim}

\subsection{Preliminaries on Maass forms}
For the sake of exposition we shall only present the case of Maass forms of weight $0$, level $P$ and nebentypus $\psi$. The case of holomorphic forms is just similar (or even simpler). Let $f:\mathbb H\rightarrow \mathbb C$ be a Hecke-Maass cusp form with Laplace eigenvalue $\frac{1}{4}+\mu^2\geq 0$, and with Fourier expansion
$$
\sqrt{y}\sum_{n\neq 0}\lambda_f(n)K_{i\mu}(2\pi|n|y)e(nx).
$$     
Let $\chi$ be a primitive Dirichlet character of modulus $M$. For simplicity we will assume that $(P,M)=~1$. The twisted $L$-series $L(s,f\otimes\chi)$, which in the right half plane $\sigma>1$ is defined by the absolutely convergent Dirichlet series $L(s,f\otimes\chi)=\sum_{n=1}^\infty \lambda_f(n)\chi(n)n^{-s}$, extends to an entire function and satisfies the functional equation $\Lambda(s,f\otimes \chi)=\varepsilon(f\otimes\chi)\Lambda(1-s,f\otimes\chi)$. The completed $L$-function is given by 
$$
\Lambda(s,f\otimes\chi)=\left(\frac{\sqrt{P}M}{\pi}\right)^s\Gamma\left(\frac{s+\delta+i\mu}{2}\right)\Gamma\left(\frac{s+\delta-i\mu}{2}\right)L(s,f\otimes\chi)
$$ 
where $\delta=0, 1$ depending on the parity of $f\otimes\chi$. The root number satisfies $|\varepsilon(f\otimes\chi)|=1$.\\

The functional equation, together with the Stirling approximation and Phragmen-Lindel\"of principle, implies the convexity bound $L\left(\tfrac{1}{2}+it,f\otimes\chi\right)\ll_{f,\varepsilon} (MT)^{\frac{1}{2}+\varepsilon}$. Alternatively the functional equation yields an expression for the $L$-values $L(\frac{1}{2}+it,f\otimes\chi)$ as a rapidly convergent series, called the approximate functional equation. Taking a dyadic subdivision of the approximate functional equation, we get the bound 
\begin{align}
\label{afe}
L\left(\tfrac{1}{2}+it,f\otimes\chi\right)\ll_{f,A} \sum_{N \;\text{dyadic}}\frac{1}{\sqrt{N}}\left|\sum_{n\in\mathbb Z}\lambda_f(n)\chi(n)n^{-it}h\left(\frac{n}{N}\right)\right|\left(1+\frac{N}{MT}\right)^{-A},
\end{align}
where $h$ is a smooth function supported in $[1,2]$, $A>0$, $t\in\mathbb R$ and $T=(3+|t|)$. Estimating the inner sums  using Cauchy and the Rankin-Selberg bound $\sum_{1\leq n\leq x}|\lambda_f(n)|^2\ll_{f,\varepsilon}x^{1+\varepsilon}$, one recovers the convexity bound $L\left(\tfrac{1}{2}+it,f\otimes\chi\right)\ll_{f,\varepsilon} (MT)^{\frac{1}{2}+\varepsilon}$. 

\subsection{Voronoi summation formula}
\label{vsf}
We will use the following Voronoi type summation formula. This was first established by Meurman \cite{Me} in the case of full level.   
\begin{lemma}
Let $f$ be as above, let $v$ be compactly supported smooth function on $(0,\infty)$, and suppose $P|q$ and $(a,q)=1$. We have
\begin{align}
\label{voronoi2}
\sum_{n=1}^\infty \lambda_f(n)e_q\left(an\right)v(n)=\frac{\bar\psi(a)}{q}\sum_{\pm}\sum_{n=1}^\infty \lambda_f(\mp n)e_q\left(\pm\bar{a}n\right)V^{\pm}\left(\frac{n}{q^2}\right)
\end{align}
where $\bar{a}$ is the multiplicative inverse of $a\bmod{q}$, and
\begin{align*}
V^-(y)=&-\frac{\pi}{\cosh \pi\mu}\int_0^\infty v(x)\{Y_{2i\mu}+Y_{-2i\mu}\}\left(4\pi\sqrt{xy}\right)dx\\
V^+(y)=&4\cosh \pi\mu\int_0^\infty v(x)K_{2i\mu}\left(4\pi\sqrt{xy}\right)dx.
\end{align*}
\end{lemma}

Note that if $v$ is supported in $[Y,2Y]$, satisfying $y^jv^{(j)}(y)\ll_j 1$, then the sums on the right hand side of \eqref{voronoi2} are essentially supported on $n\ll q^2(qY)^{\varepsilon}/Y$ (where the implied constant depends on the form $f$ and $\varepsilon$). The contribution from the terms with $n\gg q^2(qY)^{\varepsilon}/Y$ is negligibly small. For smaller values of $n$ we will use the trivial bound $V^{\pm}(n/q^2)\ll Y$.

\subsection{Circle method}
We will be using a variant of the circle method, with overlapping intervals, which has been investigated by Jutila (\cite{J-1}, \cite{J-2}). For any set $S \subset \mathbb R$, let $\mathbb I_S$ denote the associated characteristic function, i.e. $\mathbb I_S(x)=1$ for $x\in S$ and $0$ otherwise. For any collection of positive integers $\mathcal Q \subset [1,Q]$ (which we call the set of moduli), and a positive real number $\delta$ in the range $Q^{-2}\ll \delta \ll Q^{-1} $, we define the function
$$
\tilde I_{\mathcal Q,\delta} (x)=\frac{1}{2\delta L}\sum_{q\in\mathcal Q}\;\sideset{}{^\star}\sum_{a\bmod{q}}\mathbb{I}_{[\frac{a}{q}-\delta,\frac{a}{q}+\delta]}(x),
$$
where $L=\sum_{q\in\mathcal Q}\phi(q)$. This is an approximation for $\mathbb I_{[0,1]}$ in the following sense: 
\begin{lemma}
\label{jutila-lemma}
We have
\begin{align}
\label{jutila}
\int_0^1\left|1-\tilde I_{\mathcal Q,\delta}(x)\right|^2dx\ll \frac{Q^{2+\varepsilon}}{\delta L^2}.
\end{align}
\end{lemma}
This is a simple consequence of the Parseval theorem from Fourier analysis.


\section{Application of circle method}
\label{setup}

We will apply the circle method directly to the smooth sum 
$$
S(N)=\sum_{n\in \mathbb Z}\lambda_f(n)\chi(n)n^{-it}h\left(\frac{n}{N}\right),
$$
which appears in \eqref{afe}. Let's recall that the function $h$ is smooth, supported in $[1,2]$, and satisfies the bound $h^{(j)}(x)\ll 1$, where the implied constant depends only on $j$. We shall approximate $S(N)$ by
$$
\tilde S(N)=\frac{1}{L}\sum_{q\in\mathcal Q}\;\sideset{}{^\star}\sum_{a\bmod{q}}\mathop{\sum\sum}_{n,m\in \mathbb Z}\lambda_f(n)\chi(m)m^{-it}e_q(a(n-m))F(n,m)
$$ 
where $e_q(x)=e^{2\pi ix/q}$, and
$
F(x,y)=h\left(\frac{x}{N}\right)h^\star\left(\frac{y}{N}\right)\frac{1}{2\delta}\int_{-\delta}^{\delta}e(\alpha(x-y))d\alpha.
$
Here $h^\star$ is another smooth function with compact support in $(0,\infty)$, and $h^\star(x)=1$ for $x$ in the support of $h$. Also we choose $\delta=N^{-1}$ so that 
$
\frac{\partial^{i+j}}{\partial^i x\partial^j y}F(x,y)\ll_{i,j} \frac{1}{N^{i+j}}.
$\\

\begin{lemma}
\label{circ}
Let $\mathcal Q\subset [1,Q]$, with $L=\sum_{q\in\mathcal Q}\phi(q)\gg Q^{2-\varepsilon}$ and $\delta=N^{-1}$. Then we have
$$
S(N)=\tilde S(N)+O_{f,\varepsilon}\left(N\frac{\sqrt{N}(QN)^{\varepsilon}}{Q}\right).
$$
\end{lemma}
\begin{proof}
Set 
$$
G(x)=\mathop{\sum\sum}_{n,m\in \mathbb Z}\lambda_f(n)\chi(m)m^{-it}h\left(\frac{n}{N}\right)h^\star\left(\frac{m}{N}\right)e(x(n-m)).
$$
Observe that $S(N)=\int_0^1G(x)dx$ and $\tilde S(N)=\int_0^1\tilde I_{\mathcal Q,\delta} (x)G(x)dx$. Hence 
$$
\left|S(N)-\tilde S(N)\right|\leq \int_0^1\left|1-\tilde I_{\mathcal Q,\delta} (x)\right|\left|\mathop{\sum}_{n\in \mathbb Z}\lambda_f(n)e(xn)h\left(\frac{n}{N}\right)\right|\left|\sum_{m\in\mathbb Z}\chi(m)m^{-it}e(-xm)h^\star\left(\frac{m}{N}\right)\right|dx.
$$
For the middle sum we have the point-wise bound $\mathop{\sum}_{n\in \mathbb Z}\lambda_f(n)e(xn)h\left(\frac{n}{N}\right)\ll_{f,\varepsilon} N^{\frac{1}{2}+\varepsilon}$. Using Cauchy we now arrive at
$$
\left|S(N)-\tilde S(N)\right|\ll_{f,\varepsilon} N^{\frac{1}{2}+\varepsilon}\left[\int_0^1\left|1-\tilde I_{\mathcal Q,\delta} (x)\right|^2dx\right]^{\frac{1}{2}}\left[\int_0^1\left|\sum_{m\in\mathbb Z}\chi(m)m^{-it}e(-xm)h^\star\left(\frac{m}{N}\right)\right|^2dx\right]^{\frac{1}{2}}.
$$
For the last sum we open the absolute value square and execute the integral. We are left with only the diagonal, which has size $N$. For the other sum we use Lemma \ref{jutila-lemma}. It follows that
$$
\left|S(N)-\tilde S(N)\right|\ll_{f,\varepsilon} \frac{(QN)^{1+\varepsilon}}{\sqrt{\delta}L}\ll_{f,\varepsilon} N\frac{\sqrt{N}(QN)^{\varepsilon}}{Q}.
$$
\end{proof}

We will choose the set of moduli in Section \ref{conclusion}. We pick the size of the moduli to be $Q=N(MT)^{\eta-\frac{1}{2}}$, so that the contribution of the error term in Lemma \ref{circ} to \eqref{afe} is bounded by $(MT)^{\frac{1}{2}-\eta+\varepsilon}$. For $N\leq (MT)^{1-2\eta}$, the trivial bound for $S(N)$ is good enough for our purpose. Now we proceed towards the estimation of $\tilde S(N)$.


\section{Estimation of $\tilde S(N)$}
\label{gl1_twists}

\subsection{Applying Poisson and Voronoi summation}
We will now assume that each member of $\mathcal Q$ is a multiple of $P$, the level of the Maass form $f$, and is coprime to $M$, the modulus of the character $\chi$. Set
\begin{align}
\label{approx-sn}
\tilde S_x(N)=\frac{1}{L}\sum_{q\in\mathcal Q}\;\sideset{}{^\star}\sum_{a\bmod{q}}\mathop{\sum\sum}_{n,m\in \mathbb Z}\lambda_f(n)\chi(m)m^{-it}e_q(a(n-m))h\left(\frac{n}{N}\right)h^\star\left(\frac{m}{N}\right)e(x(n-m))
\end{align} 
so that $\tilde S(N)=(2\delta)^{-1}\int_{-\delta}^{\delta}\tilde S_x(N)dx$. 

\begin{lemma}
\label{sum-form}
We have
\begin{align}
\label{sum-form-eq}
\tilde S_x(N)=\tfrac{N^{1-it}\psi(M)\varepsilon_{\chi}}{\sqrt{M}L}\sum_{q\in\mathcal Q}\tfrac{\chi(q)}{q}\;\sum_{\pm}\sum_{n=1}^\infty\mathop{\sum}_{\substack{m\in \mathbb Z\\(m,q)=1}}\lambda_f(\mp n)\bar\chi(m)\bar\psi(m)e_q(\pm M\bar mn)H_x^\star(m;q)V_x^{\pm}\left(\tfrac{n}{q^2}\right),
\end{align}
where $\varepsilon_{\chi}$ is the sign of the Gauss sum associated to $\chi$, $H_x^\star$ is defined in \eqref{h-trans}, and 
$V_x^{\pm}$ are as in \eqref{voronoi2} corresponding to $v_x(n)=h\left(\frac{n}{N}\right)e(xn)$.
\end{lemma}

\begin{proof}
First we apply the Poisson summation formula to the sum over $m$ in \eqref{approx-sn}, after breaking it up modulo $Mq$. This gives
$$
\mathop{\sum}_{m\in \mathbb Z}\chi(m)m^{-it}e_q(-am)h^\star\left(\frac{m}{N}\right)e(-xm)=\frac{N^{1-it}\varepsilon_M}{\sqrt{M}}\mathop{\sum}_{\substack{m\in \mathbb Z\\m\equiv Ma\bmod{q}}}\bar\chi(m)\chi(q)H_x^\star(m;q),
$$
where
\begin{align}
\label{h-trans}
H_x^\star(m;q)=\int_{\mathbb R}h^\star(y)y^{-it}e(-xyN)e_{Mq}(-mNy)dy.
\end{align}
To the sum over $n$ we apply Voronoi summation formula \eqref{voronoi2} to get
$$
\sum_{n=1}^\infty \lambda_f(n)e_q(an)h\left(\frac{n}{N}\right)e(xn)=\frac{\bar\psi(a)}{q}\sum_{\pm}\sum_{n=1}^\infty \lambda_f(\mp n)e_q\left(\pm\bar{a}n\right)V_x^{\pm}\left(\frac{n}{q^2}\right),
$$
where $V_x^{\pm}$ are as in \eqref{voronoi2} corresponding to $v_x(n)=h\left(\frac{n}{N}\right)e(xn)$. 
\end{proof}

\subsection{Estimates for the integrals}
By repeated integration by parts it follows that $H_x^\star(m;q)$, as given in \eqref{h-trans}, is negligibly small if $|m|\gg MTQ^{1+\varepsilon}N^{-1}$, where $T=3+|t|$. For smaller values of $m$, we change variables to get
\begin{align}
\label{h-trans-est}
H_x^\star(m;q)=&\frac{|m|^{it}}{m}\left(\frac{Mq}{N}\right)^{1-it}\int_{\mathbb R}h^\star\left(\frac{Mqy}{mN}\right)e\left(-\frac{Mqxy}{m}\right)|y|^{-it} e(-y)dy\\
\nonumber =&\frac{|m|^{it}}{m}\left(\frac{Mq}{N}\right)^{1-it}H_x^\sharp(m;q).
\end{align}
Differentiating under the integral sign and using the second derivative bound for the exponential integral we get 
\begin{align}
\label{h-trans-est-2}
u^j\frac{\partial^j}{\partial u^j}H_x^\sharp(u;q)\ll_j \frac{|u|N}{Mq\sqrt{T}}
\end{align} 
(as $|x|\ll N^{-1}$). For the $n$-sum in Lemma \ref{sum-form}, we use the properties of the integral $V_x^{\pm}$ which we have noted in Section \ref{vsf}. The effective support of the $n$-sum is given by $1\leq |n|\ll Q^{2+\varepsilon}/N\ll (MT)^{2\eta}Q^{\varepsilon}$. (So the $n$-sum is short and the $m$-sum is relatively long.) For small $n$ we will use the trivial bound  
$$
V_x^{\pm}\left(\frac{n}{q^2}\right)\ll N.
$$\\

We will now use these bounds to show that the contribution from small $m$ in \eqref{sum-form-eq} is good enough for our purpose. Let
\begin{align*}
\tilde S_x(N;X)=\frac{N^{1-it}\psi(M)\varepsilon_{\chi}}{\sqrt{M}L}\sum_{q\in\mathcal Q}\frac{\chi(q)}{q}\;\sum_{\pm}\sum_{n=1}^\infty\mathop{\sum}_{\substack{|m|\sim X\\(m,q)=1}}\lambda_f(\mp n)\bar\chi(m)\bar\psi(m)e_q(\pm M\bar mn)H_x^\star(m;q)V_x^{\pm}\left(\frac{n}{q^2}\right),
\end{align*}
where $|m|\sim X$ means that $X\leq |m|<2X$. 
\begin{lemma}
We have
\begin{align*}
\tilde S_x(N;X)\ll \frac{NX}{\sqrt{MT}}Q^{\varepsilon}.
\end{align*}
\end{lemma}

This lemma yields the bound $\tilde S(N)\ll N(MT)^{\eta}$. Hence for larger value of $X$ any further saving will yield subconvexity. One way will be to appeal to the large sieve inequality of Duke, Friedlander and Iwaniec for Kloosterman fractions. But a more interesting and fruitful way will be to use the flexibility of the set $\mathcal Q$ to build a bilinear structure in the circle method itself. 


\section{Estimation of $\tilde S(N)$ : conclusion}
\label{conclusion}

\subsection{Applying Cauchy and Poisson}
We choose the set of moduli $\mathcal Q$ to be the product set $P\mathcal Q_1 \mathcal Q_2$, where $\mathcal Q_i$ consists of primes in the dyadic segment $[Q_i,2Q_i]$ (and not dividing $PM$) for $i=1,2$, and $Q_1Q_2=Q=N(MT)^{\eta-\frac{1}{2}}$. Also we pick $Q_1$ and $Q_2$ (whose optimal sizes will be determined later) so that the collections $\mathcal Q_1$ and $\mathcal Q_2$ are disjoint. Using Cauchy we have
$$
\tilde S_{x}(N;X)\ll Q^{\varepsilon}\frac{\sqrt{MQ_2}}{Q^2}\left\{\sum_{q_2\in\mathcal Q_2}\;\sum_{\pm}\tilde S_{x}(N;X,q_2,\pm)\right\}^{\frac{1}{2}}+N^{-A}.
$$ 
where $A>0$, and $\tilde S_{x}(N;X,q_2,\pm)$ is given by
$$
\mathop{\sum}_{\substack{m\in \mathbb Z\\(m,Pq_2)=1}}\frac{W\left(\frac{m}{X}\right)}{m}\left|\sum_{\substack{q_1\in\mathcal Q_1\\(m,q_1)=1}}\chi(q_1)\sum_{1\leq n\ll \frac{Q^2}{N}Q^{\varepsilon}}\lambda_f(\mp n)e_{Pq_1q_2}(\pm M\bar mn)H_x^\sharp(m;Pq_1q_2)V_x^{\pm}\left(\frac{n}{Pq_1^2q_2^2}\right)\right|^2
$$
with $(MT)^{\frac{1}{2}-\eta}Q^{-\varepsilon}\ll X\ll (MT)^{\frac{1}{2}+\eta}Q^{\varepsilon}$ and $W$ is non-negative smooth function supported in $[-2,2]$ such that $W(x)=1$ for $x\in [-1,1]$. Opening the absolute square and interchanging the order of summations we get  
\begin{align}
\label{final}
\tilde S_{x}(N;X,q_2,\pm)=\mathop{\sum\sum}_{q_1, q_1'\in\mathcal Q_1}\chi(q_1\overline{q_1'})\mathop{\sum\sum}_{1\leq n, n'\ll \frac{Q^2}{N}Q^{\varepsilon}}\lambda_f(\mp n)\bar \lambda_f(\mp n')V_x^{\pm}\left(\frac{n}{Pq_1^2q_2^2}\right)V_x^{\pm}\left(\frac{n'}{Pq_1'^2q_2^2}\right)\mathcal T
\end{align} 
where
$$
\mathcal T=\sum_{\substack{m\in\mathbb Z\\(m,Pq_1q_1'q_2)=1}}e_{Pq_1q_2}(\pm M\bar mn)e_{Pq_1'q_2}(\mp M\bar mn')H_x^\sharp(m;Pq_1q_2)\bar H_x^\sharp(m;Pq_1'q_2)\frac{1}{m}W\left(\frac{m}{X}\right).
$$
Now again we apply Poisson to the sum over $m$. With this we arrive at
\begin{lemma}
We have
\begin{align*}
\mathcal T=\frac{1}{Pq_1q_1'q_2}\sum_{m\in\mathbb Z}S\left(\mp M(q_1n'-q_1'n),m;Pq_1q_1'q_2\right)\mathcal I
\end{align*}
where $S\left(\mp M(q_1n'-q_1'n),m;Pq_1q_1'q_2\right)$ is the Kloosterman sum, and 
\begin{align*}
\mathcal I=\int_{\mathbb R}y^{-1}W\left(y\right) H_x^\sharp(yX;Pq_1q_2)\bar H_x^\sharp(yX;Pq_1'q_2)e\left(-\frac{Xm}{Pq_1q_1'q_2}y\right)dy.
\end{align*}
\end{lemma}
Using \eqref{h-trans-est-2} and repeated integration by parts we find that $\mathcal I$ is negligibly small unless $|m|\ll Q_1Q^{1+\varepsilon}X^{-1}$. For smaller values of $m$ we use the bound
$$
\mathcal I\ll \frac{X^2N^2}{M^2Q^2T}\asymp \frac{X^2}{M(MT)^{2\eta}},
$$
which again follows from \eqref{h-trans-est-2}.

\subsection{Final estimates}
Using the Weil bound for the Kloosterman sums we get
$$
\mathcal T\ll \frac{X^2}{M(MT)^{2\eta}Q_1\sqrt{Q_2}}\sum_{|m|\ll \frac{Q_1Q}{X}Q^{\varepsilon}}\left(M(q_1n'-q_1'n),m,q_1q_1'q_2\right)^{\frac{1}{2}}.
$$
Suppose $Q_1\gg (MT)^{2\eta}Q^{\varepsilon}$. Since $\mathcal Q_1$ consists of primes (not dividing $M$), it follows that $q_1|(q_1n'-q_1'n)$ if and only if $q_1=q_1'$, as $1\leq |n|\ll (MT)^{2\eta}$ and $q_1$ is a prime of larger size. Also taking $Q_1\ll (MT)^{\frac{1}{4}-\frac{\eta}{2}}Q^{-\varepsilon}$, we can guarantee that each $q_2\in\mathcal Q_2$ is a large enough prime, so that $q_2|m$ if and only if $m=0$. Also as $N>(MT)^{1-2\eta}$ and $Q_1Q_2=N(MT)^{\eta-\frac{1}{2}}$, it follows from the bound $Q_1\ll (MT)^{\frac{1}{4}-\frac{\eta}{2}}Q^{-\varepsilon}$ that $Q_2>Q_1(MT)^{2\eta}Q^{\varepsilon}$. Hence $q_2|(q_1n'-q_1'n)$ if and only if $q_1=q_1'$ and $n=n'$. Using these observations we first conclude that
\begin{align}
\label{sum-t}
\mathcal T\ll \frac{XQ_1\sqrt{Q_2}}{M(MT)^{2\eta}}+\frac{X^2}{M(MT)^{2\eta}Q_1\sqrt{Q_2}}\left((q_1n'-q_1'n),q_1q_1'q_2\right)^{\frac{1}{2}},
\end{align}
where the first term is the contribution of non-zero $m$, and the last term accounts for $m=0$. Also it follows that the gcd $\left((q_1n'-q_1'n),q_1q_1'q_2\right)=1$ if $q_1\neq q_1'$, $\left((q_1n'-q_1'n),q_1q_1'q_2\right)=q_1$ if $q_1=q_1'$ but $n\neq n'$, and $\left((q_1n'-q_1'n),q_1q_1'q_2\right)=q_1^2q_2$ if $q_1= q_1'$ and $n=n'$.\\

Next we substitute the bound for $\mathcal T$ in \eqref{final}. We use Cauchy and the bound $\sum_{1\leq n\leq x}|\lambda_f(n)|^2\ll x^{1+\varepsilon}$ for the Fourier coefficients to conclude -
$$
\tilde S_{x}(N;X)\ll \left(\frac{N^{\frac{3}{2}}(MT)^{\eta}}{\sqrt{MT}Q_2^{\frac{1}{4}}}+\frac{\sqrt{MNT}}{\sqrt{Q_1}}\right)Q^{\varepsilon}+N^{-A}.
$$ 
The first term on the right hand side accounts for the contribution of the first term in \eqref{sum-t}, and the second term comes from the second term in \eqref{sum-t}. For any given $\eta$, the optimum choice of $Q_1$ is obtained by equating the two terms and using the relation $Q_1Q_2=N(MT)^{\eta-\frac{1}{2}}$. It follows that 
$$
Q_1=\frac{(MT)^{\frac{7}{6}}}{N(MT)^{\eta}}.
$$
This satisfies our requirement that $(MT)^{\frac{1}{4}-\frac{\eta}{2}}Q^{-\varepsilon}\gg Q_1\gg (MT)^{2\eta}Q^{\varepsilon}$ if $\eta<\frac{1}{18}$. To get the optimal value of $\eta$ we compare this bound with the bound for the error term in Lemma \ref{circ}, i.e. 
$$
\frac{\sqrt{MT}N}{(MT)^{\frac{7}{12}-\frac{\eta}{2}}}=\sqrt{N}(MT)^{\frac{1}{2}-\eta}.
$$ 
It follows that $\eta=\frac{1}{18}-\varepsilon$ is the optimal choice. This completes the proof of the theorem.


\end{document}